\UseRawInputEncoding
\documentclass{amsproc}
\usepackage{amssymb,amsmath,amsthm,amscd}
\newtheorem{theorem}{Theorem}
\newtheorem{lemma}{Lemma}
\newtheorem{proposition}{Proposition}

\newtheorem{corollary}{Corollary}
\newtheorem{notation}{Notation}

\begin{document}

\title[Free Prounipotent Differential Galois Group]{The Differential Galois Group of the Maximal Prounipotent Extension is Free}
\author{Andy~R. Magid}
\address{Department of Mathematics\\
        University of Oklahoma\\
        Norman OK 73019\\
 }
 \email{amagid@ou.edu}

\maketitle

\begin{abstract}
Let $F$ be a characteristic zero differential field with algebraically closed constant field, and consider the compositum $F_u$ of all Picard--Vessiot extensions of $F$ with unipotent differential Galois group. We prove that the group of $F$ differential automorphisms of $F_u$ is a free prounipotent group.
\end{abstract}

\section*{Introduction}
Throughout, $F$ denotes a characteristic zero differential field with derivation $D$ and algebraically closed field of constants $C$. The compositum $F_u$ of all Picard--Vessiot extensions of $F$ with unipotent differential Galois group is a (generally infinite) differential Galois extension of $F$ whose (pro)unipotent differential Galois group we denote by $U\Pi(F)$. We show that this group is  free prounipotent.

In fact, what we will show is that $U\Pi(F)$ is projective. In \cite[Prop. 2.8, p.86]{ml} it is shown that projective prounipotent groups are free. The converse is also true, as will be shown in Section \ref{S:projfree} below.  
Recall that a proalgebraic group $P$ is \emph{projective} in the category of proalgebraic groups if for every surjective homomorphism $\alpha: A \to B$ of proalgebraic groups and for every homomorphism $f: P \to B$ of proalgebraic groups there is a homomorphism $\phi: P \to A$ of proalgebraic groups such that $f=\alpha \circ \phi$ \cite[Definition 8, p. 29]{blmm}.
(Note: the definition in \cite{blmm} said ``epimorphism" instead of "surjective". It is clear from the context that ``surjective" was meant. In the category of (pro)algebraic groups epimorphisms are not necessarily surjective, so that the definition of projective using epimorphism is far more restrictive than that using surjective.) A prounipotent group $U$ is projective in the category of prounipotent groups provided it satisfies the above definition where $A$ and $B$ are restricted to be prounipotent. By \cite{ml} (see below), to test the projectivity, and hence freeness, of a prounipotent group $U$ it suffices to consider the case of $\alpha$'s where both $A$ and $B$ are unipotent and the kernel of $\alpha$ is isomorphic to $\mathbb G_a$. We can, moreover, assume $f$ is surjective.

By the preceding, to see that the prounipotent group $U\Pi(F)$ is projective, we need to show that for any surjection $\alpha:A \to B$ of unipotent groups with kernel $K$ isomorphic to $\mathbb G_a$ and any surjective homomorphism $f:U\Pi(F) \to B$  there is a homomorphism $\phi:U\Pi(F) \to A$ such that $f=\alpha \circ \phi$. If $\alpha$ has a splitting, namely if there is a $\beta: B \to A$  such that $\alpha \circ \beta=\text{id}_B$, then we can take $\phi = \beta \circ f$. Hence we can concentrate on the case that $\alpha$ is not split. In the non-split case, if there is a $\phi$ it must be surjective. 
In other words, to see that $U\Pi(F)$ is projective we must show that for a non-split homomorphism of unipotent groups $\alpha: A \to B$ with kernel $K$ isomorphic to $\mathbb G_a$ and surjection $f:U\Pi(F) \to B$ there is a surjection $\phi :U\Pi(F) \to A$ such that $f=\alpha \circ \phi$. 
We can of course assume that $B=A/K$. By Galois theory, a surjection $U\Pi(F) \to B$ means we have a Picard--Vessiot extension $E_B$ of $F$ with differential Galois group $B$, and a surjection $U\Pi(F) \to A$ means we have a Picard--Vessiot extension $E_A$ of $F$ with differential Galois group $A$. Thus the existence of $\phi$ amounts to starting with a Picard-Vessiot extension $E_B$ of $F$ with Galois group $A=B/K$  and finding a Picard--Vessiot extension $E_A$ of $F$ with Galois group $A$ which contains $E_B$ such that $E_B=(E_A)^K$. In Galois theory this is known as the embedding problem.

Thus proving that $U\Pi(F)$ is projective amounts to a solution of the embedding problem for extensions of unipotent groups by 
$\mathbb G_a$; this is the content of our first main result, Theorem \ref{T:UPFprojective} below. As noted, this implies that $U\Pi(F)$ is free prounipotent. This is refined in our second main result, Theorem \ref{generators}, were it is shown that $U\Pi(F)$ is free prounipotent on a set of cardinality equal to the $C$ vector space dimension of $F/D(F)$.

The group $\Pi(F)$ of $F$ differential automorphisms of the compositum of all Picard--Vessiot extensions of $F$ is a proalgebraic group whose maximal prounipotent  quotient is $U\Pi(F)$. If $\Pi(F)$ is projective (a very strong property: this implies all embedding problems over $F$ are solvable) then so is $U\Pi(F)$. Bachmayr, Harbater, Hartmann, and Wibmer \cite{bhhw} have shown that $\Pi(F)$ is free, and hence projective, in some cases.

A preliminary version of this work was originally presented at the conference ``Galois Groups and Brauer Groups" held in honor of Jack Sonn.

\section{Embedding Problem} \label{S;embeggings}

We retain the notation from the introduction: $F$ denotes a characteristic zero differential field with algebraically closed field of constants $C$. Its derivation is denoted $D_F$, with the subscript sometimes omitted. 

As noted in the introduction, to prove that $U\Pi(F)$ is projective we need to solve an embedding problem which starts with a Picard--Vessiot extension $E$ of $F$ with unipotent differential Galois group $B$ and a non-split unipotent extension $A$ of $B$ by $\mathbb G_a$. In this context the Picard--Vessiot ring of $E$ is shown in Proposition \ref{nounipotentforms} to be isomorphic to $F[B]$ (hence a polynomial ring over $F$ and \emph{a fortiori} a UFD) and the surjection $A \to B$ is split as varieties. We are going to show in Theorem \ref{embeddingthm} that this embedding problem has a solution when the hypotheses are weakened to only require that $B$ is a proalgebraic group such that $F[B]$ is a UFD with a $B$ invariant derivation extending $D_F$ such that the quotient field $F(B)$ has no constants except $C$. This makes $F(B)$ a possibly infinite Picard--Vessiot extension of $F$ with differential Galois group $B$. We further require that the differential ring has no non-trivial principal differential ideals. Then we solve the embedding problem when $A \to B$ is a non-split extension of proalgebraic groups  with kernel isomorphic to $\mathbb G_a$ which is split as a surjection of provarieties.

We  fix the following notation for the group $\mathbb G_a$:

\begin{notation} \label{Ga}
\[
\mathbb G_a=\{ z^a | a \in C\} \text{ with } z^az^b=z^{a+b}
\]
\[
C[\mathbb G_a]=C[y] \text{ with } y(z^a)=a
\]

The action of $\mathbb G_a$ on $C[\mathbb G_a]$ (left action on functions from right translation action on the group) is then given by

\[
z^b\cdot y(z^a) = y(z^az^b)=a+b=y(z^a)+b
\]
\[
z^b \cdot y = y+b
\]

\end{notation}

We also introduce some notational conventions for extensions by $\mathbb G_a$ which are split as varieties:

\begin{notation} \label{extensionofGa} 

Let
\[
1 \to \mathbb G_a \to G \to \overline{G} \to 1
\]
be a central extension of (pro)algebraic groups over $C$ which
splits as varieties. 
 
Denote  the map $G \to \overline{G}$ by  $g \mapsto \overline{g}$.

Denote the variety section  $ \overline{G} \to G$ by  $\psi$ so that $\overline{\psi({\overline{g})}}=\overline{g}$.

Then $\phi(g):=\psi(\overline{g})$ can be regarded as a function on $ G$

Taking $\mathbb G_a$ to be a subgroup of $G$ and using the conventions of Notation \ref{Ga}, we define the function $y \in C[G]\subset F[G]$ by
\[
g=\phi(g)z^{y(g)}.
\]
We call this the $y-\phi$ representation of elements of $G$. Then
\[
C[G]=C[\overline{G}][y];  \text{ and   } C[\overline{G}]=C[G]^{\mathbb G_a}.
\]
\end{notation}

With these conventions, we have the following solution of some Embedding Problems for central extensions of $\mathbb G_a$ which are split as varieties but not as groups. The result has a statement about factorality and units in both its hypotheses and conclusions; this is to enable the result to be used inductively. Note that, in the notation of the statement of Theorem  \ref{embeddingthm}, $F(G) \supset F$ and $F(\overline{G}) \supset F$ have no new constants, so they are Picard--Vessiot extensions with groups $G$ and $\overline{G}$, respectively, and $F(\overline{G}) \subset F(G)$, solving the associated embedding problem.

\begin{theorem} \label{embeddingthm} Let $\overline{G}$ be a (pro)algebraic group over $C$. Assume that $F[\overline{G}]$ is a unique factorization domain. Assume further that there is a derivation $D$ of $F[\overline{G}]$ extending $D_F$ such that
\begin{enumerate}
\item $F(\overline{G})$ has no new constants \\
\item If $0 \neq q \in F[\overline{G}]$ and $q|D(q)$ then $q$ is a unit\\
\end{enumerate}
Let
\[
1 \to \mathbb G_a \to G \to \overline{G} \to 1
\]
be a central extension of $C$ groups which
splits as varieties but not as algebraic groups over $F$. Then there is a derivation on $F[G]$ extending $D$ which commutes with the $G$ action and is such that $F(G)$ has no new constants. Moreover, $F[G]$ is a unique factorization domain and if $0 \neq q \in F[G]$ and $q|D(q)$ then $q$ is a unit.
\end{theorem}

\begin{proof} Using the  conventions of Notations \ref{Ga} and \ref{extensionofGa}, and the fact that $\mathbb G_a$ is central,we compare the product of $y-\phi$ representations with the $y-\phi$ representation of the product:
\begin{align}\notag
gh&=\phi(gh)z^{y(gh)} \notag\\
gh&=\phi(g)z^{y(g)}\phi(h)z^{y(h)}\notag\\
&=\phi(g)\phi(h)z^{y(g)+y(h)}\notag\\ \notag
\end{align}
Then we combine to define the function $\alpha$ on $G \times G$, which, since $\phi(g)=\psi(\overline{g})$, can also be viewed as a function on  $\overline{G} \times \overline{G}$:
\[
z^{y(gh)-y(g)-y(h)}=\phi(gh)^{-1}\phi(g)\phi(h)=:z^{\alpha(g,h)} 
\]
We can then use $\alpha$ to describe the action of $G$ on $y$:
\begin{align}\notag
y(gh)&=y(g)+y(h)+\alpha(g,h)\notag\\
h \cdot y&= y+y(h)+\alpha(\cdot,h)\notag\\ \notag
\end{align}

The second of these equations implies that $\alpha(\cdot,h) = h \cdot y - y -y(h)$ is an element of $C[\overline{G}] \subseteq F[\overline{G}]$. Thus $\alpha(\cdot, h)$ is a $C$ valued function on $\overline{G}$.
As noted above in Notation \ref{extensionofGa}, $F[G]=F[\overline{G}][y]$. Since $D$ needs to extend the derivation on $F[\overline{G}]$, we need only define it on $y$. 
Thus to define $D$ it will suffice to set $D(y)=f \in F[\overline{G}]$ for some appropriate element $f$. For $D$ to be $G$ equivariant, we want $D(h\cdot y)=h\cdot D(y)$, which by the above means $f+D(\alpha(\cdot,h))=h\cdot f$. Note that this is a condition on $f$.

Thinking of $\alpha$ as a $C$ valued function on $\overline{G} \times \overline{G}$, we define:
\[
\sigma(h):=D(\alpha(\cdot,h)) \text{ (a function on } \overline{G}).
\]
We are going to show that $\sigma$ is a cocycle. We calculate: $\sigma(hk)$, using $x$ to stand for the variable argument also symbolized by $(\cdot)$:
\begin{align}
\sigma(hk)(x)&=D(\alpha(x,hk))=D(y(x(hk))-y(x)-y(hk))=D(y(x(hk))-y(x))-D(y(hk))\notag\\
&=D(y(x(hk))-y(x)) \text{ ($y$ is $C$ valued so $D(y(hk))=0$); and }\notag\\
\sigma(h)(x)+h\cdot\sigma(k)(x)&=D(\alpha(x,h))+D(\alpha(xh,k))=D(\alpha(x,h)+\alpha(xh,k))\notag\\
 &=D( y(xh)-y(x)-y(h)+y((xh)k)-y(xh)-y(k))\notag\\
 &=D(y((xh)k)-y(x))-D(y(h)+y(k))\notag\\
 &=D(y(xh)k)-y(x))\notag\\ \notag
 \end{align}
 Thus
 \[
\sigma(hk)=\sigma(h)+h\cdot \sigma(k)
\]
so 
\[
\sigma \in Z^1(\overline{G},F[\overline{G}])
\]

In \cite[Proposition 2.2, p. 495]{h}, it is shown that (1) $C[\overline{G}]$ (and therefore $F[\overline{G}]$) is an injective $\overline{G}$ module; and (2) for any $\overline{G}$ module $M$, $\text{Ext}^1_{\overline{G}}(C,M)=Z^1(\overline{G},M)/B^1(\overline{G},M)$. For $M=F[\overline{G}]$, then, every cocycle is a coboundary.
It follows that $\sigma =\delta(f)$ for some $f \in F[\overline{G}]$. (Note: the results of \cite{h} are for linear algebraic groups; the extensions to proalgebraic groups are straightforward.)

We use the $f$ such that $\sigma =\delta(f)$ (so that $\sigma(h)=h\cdot f -f$) in the definition of $D$.
Since by definition $\sigma(h)=D(\alpha(\cdot,h))$, we have $D(\alpha(\cdot,h))=h\cdot f -f$, or $f+D(\alpha(\cdot,h))=h\cdot f$. This is precisely the condition obtained above for the $G$ invariance of $D$.

Thus $D$ extends the derivation of $F[\overline{G}]$.

Next, we want to show that $F(G)$ has no new constants. We can regard $F(G)$ as the quotient field of $F(\overline{G})[y]$. Since $F(\overline{G})$ has no new constants, we claim that $F(G)$ has no new constants provided that $f$ is not a derivative in $F(\overline{G})$. This is an elementary direct calculation; for example, see \cite[Remark 1.10.2 p.7]{m}. 

It remains to show that $f$ is a not a derivative. Suppose it is. Since $F[\overline{G}]$ is a UFD, if $f$ is the derivative of an element of the quotient field of $F[\overline{G}]$, $f=D(p/q)$ where $p$ and $q$ are relatively prime elements of $F[\overline{G}]$. Then
\[
fq^2=qD(p)-pD(q)
\]
which implies that $q|D(q)$. By assumption, this means that $q$ is a unit of $F[\overline{G}]$ and that $f$ is the derivative of $p/q \in F[\overline{G}]$.

Let $f_0$ denote $p/q$. Replace $f_0$ by $f_0-f_0(e)$ so that $f_0(e)=0$.

Then $D(\alpha(\cdot,h)) =h\cdot D(f_0) -D(f_0)= D(h\cdot f_0 -f_0)$ which means
$\alpha(\cdot,h)=h\cdot f_0 -f_0 -c_h$ for some $c_h \in C$.

Since $\alpha(g,h)=f_0(gh)-f_0(g)-c_h$ and  $0=\alpha(e,h)=f_0(h)-f_0(e)-c_h$, $c_h=f_0(h)$.
Thus $y(gh)=y(g)+y(h)+\alpha(g,h)=y(g)+y(h)+f_0(gh)-f_0(g)-f_0(h)$. 

Since $y \in C[G] \subseteq F[G]$ and $f_0 \in F[\overline{G}] \subseteq F[G]$, their difference $y-f_0$ lies in $F[G]$.

Let $x=y-f_0$. Then $x(gh)=y(gh)-f_0(gh)=y(g)+y(h)+f_0(gh)-f_0(g)-f_0(h)-f_0(gh)=x(g)+x(h)$.
which implies that $x$ is a homomorphism.

Since $x$ is a homomorphism and $F[G]=F[\overline{G}][y]=F[\overline{G}][x]$,
\[
1 \to (\mathbb G_a)_{ F} \to G_{ F} \to (\overline{G})_{F} \to 1 \text{ splits as a group extension. }
\]
Since $C$ is algebraically closed, this means that the extension already splits as groups over $C$.

This contradiction means $f$ is not a derivative in the quotient field of $F[\overline{G}]$. We conclude that $F(G)$ has no new constants.

We also observe that $F[G]=F[\overline{G}][y]$ is a unique factorization domain. To complete the proof, suppose $0 \neq q \in F[G]$ and $q|D(q)$. We can write $q$ as a polynomial in $y$ with coefficients in $F[\overline{G}]$, say $q=\sum_{k=0}^n a_ky^k$ with $a_n \neq 0$. Then $D(q)=\sum (D(a_k)y^k+ka_ky^{k-1})f$ has degree at most $n$. Thus $D(q)=bq$ for some $b$ in $F[\overline{G}]$. In particular, $D(a_n)=ba_n$. Since $a_n|D(a_n)$, this means $a_n$ is a unit. We differentiate $q/a_n$:
\[
D(\frac{q}{a_n})=\frac{a_nD(q)-qD(a_n)}{a_n^2})=\frac{q}{a_n}\frac{a_nb-D(a_n)}{a_n} \text{  so}
\]
\[
\frac{q}{a_n}|D(\frac{q}{a_n}).
\]

So we can replace $q$ by $q/a_n$ and hence we can assume $a_n=1$. Then $D(q)$ has degree less than $n$, so $q|D(q)$ implies that $D(q)=0$. Thus $q$ is a constant of $F(G)$, and we know these are in $C$. In particular, $q$ is a unit of $F[G]$.

This completes the proof of the theorem.
\end{proof}

\section{Projective Galois Groups} \label{S:projgroups}

We are going to apply Theorem \ref{embeddingthm} when the group $\overline{G}$ is (pro)unipotent, and hence so is $G$. This implies  that  $F[G]$ is a polynomial ring, and in particular a unique factorization domain, all of whose units are in $F$.

With that application in mind, we make an observation about (infinite) Picard--Vessiot extensions whose differential Galois group is (pro)unipotent.

\begin{proposition} \label{nounipotentforms} Let $E \supset F$ be a (possibly infinite) Picard--Vessiot extension with (pro)unipotent differential Galois group $H$. Then its Picard--Vessiot ring $R$ is isomorphic to $F[H]$ as a ring and an $H$ module.
\end{proposition}

Before proving the proposition we make some preliminary remarks. Proposition \ref{nounipotentforms} is a prounipotent version of Kolchin's Theorem,  \cite[Theorem 5.12 p. 67]{m}, which says that if the affine algebraic group $G$ is the differential Galois group of the Picard--Vessiot extension $K \supseteq F$ with Picard--Vessiot ring $T=T(K/F)$ then there is a $G$ equivariant $\overline{F}$ algebra isomorphism
\[
\overline{F} \otimes_F T \simeq \overline{F} \otimes_F F[G]=\overline{F}[G].
\]
The isomorphism is given explicitly as $h=(f \otimes 1)\circ \Delta$, where $\Delta: T \to T \otimes_C C[G]$ is action of $G$ on $T$, and $f:T \to \overline{F}$ is any $F$ algebra homomorphism. (Note that the facts that $f$ exists, and that the proof that $h$ is an isomorphism, use that $T$ is an affine $F$ algebra.) Suppose further that there is an $F$ algebra homomorphism $f_0: T \to F$. Then we can consider the map $h_0=(f_0 \otimes 1)\circ \Delta: T \to F[G]$. Since $h=\overline{F} \otimes  h_0$ is an isomorphism, $h_0$ is also.

Now suppose $K \supseteq F$ is an infinite Picard--Vessiot extension,  $G$ its proaffine differential Galois group, and $T$ its Picard--Vessiot ring. Suppose further that there is an $F$ algebra homomorphism $T \to \overline{F}$. Then we can define the $\overline{F}$ algebra $G$ equivariant homomorphism $h=(f \otimes 1) \Delta$ as above. For any closed normal subgroup $H$ of $G$ the restriction of $h$ to $\overline{F} \otimes_F T(G)^H=\overline{F} \otimes_F T(G/H)$, which maps to 
$\overline{F}[G]^H=\overline{F}[G/H]$, is an isomorphism by Kolchin's Theorem applied to the Picard--Vessiot extension $K^H \supseteq F$ with affine differential Galois group $G/H$. Since $T$ is the union of the $T^H$ as $H$ ranges over all closed normal subgroups of $G$, and the restrictions of $h$ to  the $T^H$'s are isomorphisms, we conclude that Kolchin's Theorem holds in the infinite Picard--Vessiot case \emph{provided there is a map} $f:T \to \overline{F}$. And just as in the affine case
(with the same proof), 
if there is an $F$ algebra homomorphism $T \to F$ then there is a $G$ equivariant $F$ algebra isomorphism $T \to F[G]$. Such an isomorphism is the conclusion of Proposition \ref{nounipotentforms}. Thus to prove Proposition \ref{nounipotentforms}, it suffices to show the existence an $F$ algebra homomorphism $T \to F$. Conversely, if the proposition holds, there is such a homomorphism, for example given by  $F[H] \to F$ by evaluation at an element of $H$. We proceed to the proof, reverting to the notation of the proposition.

\begin{proof} (of Proposition \ref{nounipotentforms}). When $H$ is unipotent, Kolchin's Theorem says that $R$ is the (coordinate ring of) an affine $F$ variety which is an $\overline{F}/F$ form of $F[H]$. These are classified by the non-commutative Galois cohomology set $H^1(F, H)$ \cite[Theorem 2.9, p.67]{platrap} and this latter set is trivial when $H$ is unipotent \cite[Lemma 2.7, p. 74]{platrap}. Thus the result holds in the unipotent case. For the prounipotent case, we begin by using Zorn's Lemma to find a closed normal subgroup $N \leq H$ minimal with respect to the property that there is an $F$ algebra homomorphism $f:T^N \to F$. If $N$ is trivial, the proposition follows. If $N$ is not trivial, then there is a proper closed normal subgroup $M < H$ with $H/M$ affine and $T^M$ not contained in $T^N$. By the unipotent case we have $T^M=T(E^M/F) \simeq F[H/M]$. The projections $H \to H/N$ and $H \to H/M$ induce an isomorphism $H/(N \cap M) \to (H/N)\times_{H/NM} (H/M)$ and hence an $H$ algebra isomorphism $\beta: F[H/(N \cap M)] \to F[H/N] \otimes_{F[H/NM]} F[H/M]$. Since the source of $\beta$ is a simple $H$ algebra, so is the target. The unipotent case further implies that $T^{NM} \simeq F[H/NM]$. The splitting of $H/M \to H/NM$ as a map of varieties, and the fact that unipotent groups are isomorphic to affine spaces as varieties, means that $F[H/M]$ is a polynomial ring over $F[H/NM]$ from which it follows that $T^M$ is a polynomial ring over $T^{NM}$
We also have $T^{NM} \subset T^N$. Let $g_0$ denote the restriction of $f$ to $T^{NM}$; since $T^M$ is a polynomial ring over $T^{NM}$ there is an $F$ algebra homomorphism $g:T^M \to F$ which agrees with $g_0$ on $T^{NM}$. Then 
$g \otimes_F f$ defines an $F$ algebra homomorphism $p: T^M \otimes_{T^{NM}} T^N \to F$. We also know that $T^M \otimes_{T^{NM}} T^N \simeq   F[H/N] \otimes_{F[H/NM]} F[H/M]$ and that the target is $H$ simple, and hence so is the source. This tells us that the $H$ equivariant $F$ algebra homomorphism $\gamma: T^M \otimes_{T^{NM}} T^N  \to T$ is injective; let $T_0$ denote its image. Note that $f \circ \gamma^{-1}: T_0 \to F$ is an $F$ algebra homomorphism. We have an $H$ equivariant $F$ algebra isomorphism $\gamma \circ \beta : F[H/(N \cap M)] \to T_0$. We claim  that $T_0=T^{N\cap M}$. If so, we have a contradiction to the minimality of $N$: $N \cap M$ is a proper subgroup of $N$ and there is an $F$ algebra homomorphism $T^{N \cap M} \to T$. 
Consider a $t \in T$ fixed by $N \cap M$. Choose a 
 closed normal subgroup $Q \leq H$ with $(N \cap M) \leq Q$,  $H/Q$ unipotent, and $t \in T^Q$. 
By the unipotent case $T^Q$ is $H$ equivariantly isomorphic to $F[H/Q]$ as an $F$ algebra. But $F[H/Q]=F[H/(N \cap M)]^Q$ is also $H$ equivariantly isomorphic to $T_0^Q$ as an $F$ algebra. Thus the inclusion $T_0^Q \to T^Q$ is an $H$ equivariant $F$ algebra endomorphism of $F[H/Q]$, and any such is in fact an algebra isomorphism. In particular, $t$ is in its image and hence in $T_0$. 
We conclude that the claim holds, and hence the proposition is proven. 
\end{proof}

Proposition \ref{nounipotentforms} implies that if  $E \supset F$ is a (possibly infinite) Picard--Vessiot extension with (pro)unipotent differential Galois group $\overline{G}$, then its Picard--Vessiot ring $F[\overline{G}]=F\otimes C[\overline{G}]$ satisfies the hypotheses of Theorem \ref{embeddingthm}. To apply the theorem to an extension of $\overline{G}$ by $\mathbb G_a$, we need to know that the extension in question is split as varieties (and not split as groups). All extensions of (pro)unipotents by (pro)unipotents are split as varieties: this fact seems to be well known, so we only sketch the proof.

\begin{proposition} \label{unipotentsplit} Let
\[
1 \to K \to G \to H \to 1
\]
be an extension of a prounipotent group $H$ by a prounipotent group $K$. Then the extension splits as varieties.
\end{proposition}

\begin{proof} When $H$ and $K$ are unipotent, so is $G$. Both $G$ and $H$ are isomorphic as varieties to their Lie algebras. A right linear inverse to the linear vector space projection $\text{Lie}(G) \to \text{Lie}(H)$ composed with these isomorphisms is a variety section. The same argument works for prounipotent groups, using the complete Lie algebras \cite[1.1 p.78]{ml}. There are (pro)variety isomorphisms of the groups to the complete Lie algebras, which are as additive groups products of copies of $\mathbb G_a$, and the surjection  between them has a linear inverse because the kernel is a closed subspace.
\end{proof}

A linear action of a prounipotent group on a one--dimensional vector space is trivial, so a normal subgoup of a prounipotent group isomorphic to $\mathbb G_a$ is central.

Now we come to our main application.

\begin{theorem} \label{alwaysextend} Let $E \supset F$ be a (possibly infinite) Picard--Vessiot extension with (pro)unipotent differential Galois group $\overline{G}$. Let
\[
1 \to \mathbb G_a \to G \to \overline{G} \to 1
\]
be a  extension which does not split as algebraic groups over $F$. Then there is a Picard--Vessiot extension $E_1 \supset F$ with differential Galois group $G$ such that $E_1 \supset E$ and such that the restriction map on differential Galois groups is the given map $G \to \overline{G}$.
\end{theorem}

\begin{proof} By Proposition \ref{nounipotentforms} we may assume $E$ is the quotient field of $F[\overline{G}]$ and so the latter satisfies the hypotheses of Theorem \ref{embeddingthm}. By Proposition \ref{unipotentsplit} the extension $1 \to \mathbb G_a \to G \to \overline{G} \to 1$ also satisfies the hypotheses of the theorem. Then $E_1=F(G)$ with the derivation of the theorem is the desired Picard--Vessiot extension.
\end{proof}

Theorem \ref{alwaysextend} applies of course when $\overline{G}$ is unipotent, and asserts that a solution of the embedding problem for extensions of unipotent groups by 
$\mathbb G_a$ always exists. We can now conclude that the differential Galois group of the compositum of the unipotent extensions of $F$ is projective.

\begin{theorem} \label{T:UPFprojective}
Let $U\Pi(F)$ be the differential Galois group of the compositum $F_u$ of all Picard--Vessiot extensions of $F$ with unipotent differential Galois group. Then $U\Pi(F)$ is a projective prounipotent group.
\end{theorem}

\begin{proof} By \cite[Theorem 2.4, p.84]{ml}, it suffices to show that for any unipotent group $B$ and any extension of $A$ of $B$ by $\mathbb G_a$ a homomorphism $f: U\Pi(F)  \to B$ can be lifted to $A$. Note that $A$ is also unipotent. Let $\alpha: A \to B$ be the projection. If $f$ is not surjective, we can replace $B$ with $B^\prime=f(U\Pi(F))$ and $A$ with $A^\prime =R_u(\alpha^{-1}(B^\prime)$ (the unipotent radical of the inverse image) to obtain an extension $A^\prime \leq A$  of $B^\prime \leq B$ by $\mathbb G_a$ and a surjective homomorphism $U\Pi(F) \to B^\prime$. If this homomorphism can be lifted to $A^\prime$ then the same homomorphism lifts $f$ to $A$. As remarked above, if $A^\prime$ is a split extension of $B^\prime$ then the splitting produces the lift. Thus we can assume $A^\prime$ is a non-split extension of $B^\prime$. We drop the ``primes" and revert to the original notation. The surjection $f$ means that we have a Picard--Vessiot extension $E_B$ of $F$ with differential Galois group $B$ and a non--split exact sequence 
\[
1 \to \mathbb G_a \to A \to B \to 1.
\]
By Theorem \ref{alwaysextend} there is a differential Galois extension $E_A \supset F$ with differential Galois group $B$ (and hence a surjection $U\Pi(F) \to B$) and $E_B \supset E_A$ (which implies that the surjection lifts $f$).
\end{proof}

As noted in the introduction, for prounipotent groups projectives are free \cite[Proposition 2.8. p. 86]{ml}. 
Free prounipotent groups, like other free objects, are free on a subset. By  \cite[Defn. 2.1, p. 83]{ml} and \cite[Propostion 2.2, p. 83]{ml}, the free prounipotent group $U(I)$ on the subset $I$ is universal with respect to the property that for unipotent groups $U$, set maps $I \to U$ with all but finitely many elements going to the identity extend uniquely to morphisms $U(I) \to U$. By \cite[Lemma 2.3, p. 84]{ml}, the cardinality of the subset $I$ is the same as the dimension of $\text{Hom}(U(I), \mathbb G_a)$.

Thus Theorem \ref{T:UPFprojective} implies that $U\Pi(F)$ is free prounipotent as we record in Corollary \ref{closurefree}. In Theorem \ref{generators} below we determine the cardinality of a set on which it is free.

\begin{corollary} \label{closurefree} Let $F_u$ be the compositum of all the unipotent Picard--Vessiot extensions of $F$. Then the differential Galois group $U\Pi(F)$ of $F_u$ over $F$ is free prounipotent.
\end{corollary}

Theorem \ref{alwaysextend} shows that if there is a $\mathbb G_a$ extension of a prounipotent differential Galois group over $F$, then the extension is realized as a differential Galois group. In \cite[Theorem 2.9, p.87]{ml}, it is shown that if a prounipotent group $G$ is not free, then $H^2(G,C) \neq 0$ (and conversely). This $H^2$  is derived functor cohomology in the category of rational $G$ modules, which by \cite{h} is also Hochschild cohomology (two cocycles modulo two coboundaries). Because of the fact, used above, that extensions of prounipotents by $\mathbb G_a$ split as varieties, Hochschild cohomology corresponds to extensions \cite[Proposition 2.3 p.190]{dg}. Thus all non-free prounipotent groups have non trivial extensions by $\mathbb G_a$, and conversely. We need to remark here that non trivial means non split as algebraic groups over $C$, whereas Theorem \ref{alwaysextend} refers to non split as algebraic groups over $F$. Actually the former implies the latter: for if $H^2(G,C)$ is non trivial, the same is true for extension to any algebraically closed field $\mathcal C$ over $C$. This follows from constructing an injective resolution of $C$ as a $G$ module whose terms are sums of the modules $C[G]$. Then tensoring this resolution over $C$ with $\mathcal C$, which is an exact functor, produces a resolution of $\mathcal C$ whose terms are sums of $\mathcal C[G]$, hence injective, and taking $G$ invariants and homology commutes with tensoring as well, so that $H^2(G_{\mathcal C}, \mathcal C)$ is isomorphic to $\mathcal C \otimes H^2(G,C)$.

It is possible that $F$ has no unipotent Picard--Vessiot extensions; for example, it may be Picard--Vessiot closed \cite{dartii}. In this situation the free prounipotent group of Corollary \ref{closurefree} is the free prounipotent group on no generators. It is possible that $F$ has a unique unipotent Picard--Vessiot extension with group $\mathbb G_a$, which is the free prounipotent group on one generator; this happens for $F=C$. In both these cases, the number of generators of the free prounipotent group is the dimension of $F/D(F)$ as a $C$ vector space. As we now show, this happens in general. We set the following notation:

\begin{notation} \label{nonderivatives}

Let $U\Pi =U\Pi(F)$ denote the differential Galois group of $F_u$ over $F$. 

Let $\{x_\alpha \in F| \alpha \in \mathcal A\}$ be such that their images in $F/D(F)$ are a basis over $C$. 

Let $y_\alpha \in F_u$ be such that $D(y_\alpha)=x_\alpha$. 

Let $G_\alpha$ be the differential Galois group of $F(y_\alpha)$ over $F$. Note that $G_\alpha$ is isomorphic to $\mathbb G_a$. 

\end{notation}

For the following lemmas, which describe some of the properties of the $x_\alpha$'s and $y_\alpha$'s, we recall some basic properties of elements of Picard--Vessiot extensions whose derivatives are in the base. Let $E \supseteq F$ be a Picard--Vessiot extension and let $z \in E$ with $D(z) \in F$. For any differential automorphisms $\sigma$ of $E$ over $F$,  $D(\sigma(z)-z)= \sigma(D(z))-D(z)=0$, since $D(z)\in F$ and $\sigma$ is trivial on $F$. So $\sigma(z) - z \in C$.  Suppose $z_i$, $1 \leq i \leq n$ are elements of $E$ with $D(z_i) \in F$ for all $i$. Then the previous observation shows that the subfield $E_0=F(z_1, \dots, z_n)$ of $E$ is stable under every differential automorphism of $E$ over $F$, which makes $E_0 \supseteq F$ a Picard--Vessiot extension. Call its differential Galois group $H$. Since the $z_i$ generate $E_0$ over $F$, the map $\phi: H \to C^n$ by $\tau \mapsto (\tau(z_1)-z_1, \dots, \tau(z_n) -z_n)$ is injective. It is also an algebraic group homomorphism, so $H$ is a commutative unipotent group. Suppose further that $z_1, \dots, z_n$ are algebraically independent over $F$. Then for any $(c_1, \dots c_n) \in C^n$ we can define an $F$ automorphism of $F(z_1, \dots, z_n)$ by $z_i \mapsto z_i + c_i$, $1 \leq i \leq n$. Since $D(z_i+c_i)=D(z_i)$, this automorphism  is differential and its image under $\phi$  is $(c_1, \dots, c_n)$. Thus $\phi$ is surjective and 
hence $\phi$ is an isomorphism.

\begin{lemma} \label{simpletranscendentalextension} Let $E$ be a Picard--Vessiot extension of $F$ with Galois group $H$. Suppose $z_i \in E$, $1 \leq i \leq n$ are  algebraically independent over $F$ such that 
\[
D(z_i) \in F \text{ for } 1 \leq i \leq n \text{ ; and}
\]
\[
E=F(z_1, \dots, z_n)
\]
Let $y \in E$ with $D(y) \in F$. Then there are $c_1, \dots, c_n \in C$ and $f \in F$ such that
\[
y= f+ \sum_1^n c_iz_i
\]
\end{lemma}

\begin{proof} As noted above, $\phi: H \to C^n$ by $\tau \mapsto (\tau(z_1)-z_1, \dots, \tau(z_n) -z_n)$ is an isomorphism. For each $i$, choose $\sigma_i \in H$ such that $\phi(\sigma)$ is the $n$-tuple with $1$ in position $i$ and $0$ elsewhere. Then the $\sigma_i$'s generate $H$ as an algebraic group. 
Let $y \in E$ with $D(y)\in F$. Then, as noted above, for each $i$, $\sigma_i(y) - y \in C$, say $\sigma_i(y)=y+c_i$.
Let $g=\sum_1^n c_jz_j$.
Then $\sigma_i(g)=g+c_i$ for each $i$, so $\sigma_i(y-g)=y-g$ for all $i$. Since the $\sigma_i$'s generate $H$, $\tau(y-g) =y-g$ for all $ \tau \in H$, which implies that $y-g=f \in F$ and hence $y=f +\sum_1^n c_jz_j$.
\end{proof}

\begin{lemma} \label{transcendance}  The elements  $y_\alpha$,  $ \alpha \in \mathcal A$, are algebraically independent over $F$.
\end{lemma}

\begin{proof} Let $y_1, \dots, y_m$ be a subset of the $y_\alpha$'s.
For any $1 \leq k \leq m$ the subfield $E_k=F(y_1, \dots, y_k)$ of $F_u$  is a Picard--Vessiot extension. Let $G_k$ be the differential Galois group of $E_k$ over $F$. If $\sigma \in G_k$ then $\sigma \mapsto (\sigma(y_1) -y, \dots , \sigma(y_k)-y_k)$ defines a homomorphism $\phi_k: G_k \to \mathbb G_a^k$. We use induction on  $m$ to prove that  $\phi_m$ is an isomorphism and that the elements $y_i, \dots, y_m$ are algebraically independent over $F$.
The case $m=1$ is trivial. So suppose the result holds for $m=k-1$ and consider the case $m=k$. Consider $E_k \supset E_{k-1}$. These are Picard--Vessiot extensions of $F$ with differential Galois groups $G_k$ and $G_{k-1}$ respectively,  $G_{k-1}$ is a product of $k-1$ copies of $\mathbb G_a$ by the induction hypothesis and $G_k$ is a subgroup of $k$ copies  of $\mathbb G_a$. Since $\psi: G_k \to G_{k-1}$ is onto, if $G_k$ is not $k$ copies $\psi$ is an isomorphism and $E_k=E_{k-1}$. In particular, $y_k \in E_{k-1}$. Since $D(y_k)=x_k \in F$,  
we conclude from Lemma \ref{simpletranscendentalextension} that $y_k=f+ \sum_1^{k-1}c_iy_i$ with $c_i \in C$ and $f \in F$. Apply $D$ to this equation. We find $x_k=D(f)+\sum_1^{k-1}c_ix_i$, 
which is a relation of linear dependence of $x_1, \dots, x_k$ modulo $D(F)$. This contradiction implies that $G_k$ is actually a product of $k$ copies of $\mathbb G_a$. The transcendence degree of $E_k$ over $F$ is equal to that of the function field $F(G_k)$, and since we now know this latter is $k$, we  know the transcendence degree of $E_k=F(y_1, \dots, y_k)$ over $F$ is $k$, so $y_1, \dots, y_k$ are algebraically independent over $F$.
\end{proof}

We next consider the subfield $F_u^{ab}=F(\{y_\alpha \vert \alpha \in \mathcal A\})$ of $F_u$ generated over $F$  the $y_\alpha$'s. (The notation will be explained below). For $\sigma \in U\Pi$, we know that $\sigma(y_\alpha)-y_\alpha \in C$, so that $F_u^{ab}$ is a (possibly infinite) Picard--Vessiot extension of $F$. Let $H$ be the (pro)algebraic differential Galois group of $F_u^{ab}$ over $F$. Just as in the case with finitely many generators, the map $\phi: H \to \prod_{\mathcal A} C$ which sends $\sigma$ to the tuple whose $\alpha$ entry is  $\sigma(y_\alpha)-y_\alpha$ is injective since the $y_\alpha$ generate $E$. Since the $y_\alpha$'s are algebraically independent over $F$ (Lemma \ref{transcendance}), given any $\mathcal A$ tuple of elements of $C$ we can define an element of $H$ as the automorphism which sends $y_\alpha$ to $y_\alpha$ plus the $\alpha^\text{th}$ entry of the tuple, which shows that $\phi$ is also surjective and thus an isomorphism. 

Prounipotent group homomorphisms $H=\prod_{\mathcal A} C \to \mathbb G_a$ are continuous in the product topology.  We note for use below that $\text{Hom}(H, \mathbb G_a)=\text{Hom}(\prod_{\mathcal A} C, \mathbb G_a)$ (prounipotent group homomorphisms) is a $C$ vector space of dimension the cardinality of $\mathcal A$.

Thus $H$ is an abelian prounipotent quotient of $U\Pi$. We will see it is the maximum such:

\begin{lemma} \label{characterbasis} Let $E \subseteq F_u$ be a (possibly infinite) Picard--Vessiot extension of $F$ inside $F_u$ such that the differential Galois group $M$ of $E$ over $F$ is abelian. Then $E \subseteq F_u^{ab}$. Consequently, the differential Galois group of $F_u^{ab}$ over $F$is the maximal abelian quotient $U\Pi^{ab}$ of $U\Pi$. Moreover, $\text{Hom}(U\Pi, \mathbb G_a)=\text{Hom}(U\Pi^{ab},\mathbb G_a)$ is a $C$ vector space of dimension the cardinality of 
$\mathcal A$.
\end{lemma}

\begin{proof} $M$ is a product of copies of $\mathbb G_a$ indexed by some set $\mathcal S$. For $s \in \mathcal S$, let $p_s:M \to C$ be the projection on the $s^\text{th}$ factor; so $p_s \in C[M]$.

For $m \in M$, $m \cdot p_s = p_s+p_s(m)$. Since $C[M]$ is a polynomial ring over $C$ in the $p_s$'s, $F[M]$ is polynomial ring over $F$ in the $p_s$'s, and with the same $M$ action. By Proposition \ref{nounipotentforms} the Picard--Vessiot ring $R$ of $E$ is isomorphic to $F[M]$ as an $F$ algebra and as a left $M$ module. Let $z_s \in R$ correspond to $p_s$. We will see that $z_s \in F_u^{ab}$ for all $s \in \mathcal S$, which will prove the first assertion. So fix $s$ and let $z=z_s$ and $p=p_s$. 
As in Notation \ref{Ga}, $M$ acts on the left on $C[M]$ by acting on the right on $M$, so 
for $m \in M$, $m\cdot z =z+p(m)$ with $p(m) \in C$. Then $m\cdot D(z)=D(m\cdot z)=D(z)+D(p(m))=D(z)$ for all $m$, so $D(z) \in E^M=F$. Let $x=D(z)$.  Then $x$ can be expressed, modulo $D(F)$, as a $C$ linear combination of $x_\alpha$'s, where the $\alpha$'s in question come from a finite subset $\mathcal F$ of $\mathcal A$. 
Renumbering, we write $x=\sum_{k=1}^nc_kx_k + D(w)$, with $z \in F$, where $x_k$ is as in Notation \ref{nonderivatives}, with $k$ replacing $\alpha$.
Then $D(z-\sum c_ky_k -w)=0$, where $y_k$ is as in Notation \ref{nonderivatives}, with $k$ replacing $\alpha$, so $z=(\sum c_ky_k)+w+c$, where $c$ is a constant. In particular, $z_s=z \in F(y_1, \dots, y_n)$ so $E \subseteq F(y_1, \dots, y_n) \subseteq F_u^{ab}$. Thus $E \subseteq F_u^{ab}$. Let $K$ be the kernel of $U\Pi \to U\Pi^{ab}$ and let $L$ be the kernel of the restriction of $U\Pi$ to $F_u^{ab}$. Let $E$ be the fixed field of $K$. Since $U\Pi/K$ is abelian, by the above $E \subseteq F_u^{ab}$, which means that $L \leq K$. On the other hand, since $U\Pi/L$ is abelian, $K \leq L$. 
Thus $K=L$ and the differential Galois group of $F_u^{ab}$ over $F$ is the maximal abelian quotient $U\Pi^{ab}$ of $U\Pi$, proving the second assertion. For the final statement, we note that, in the notation of the discussion preceding the Lemma, the differential Galois group $H$ of $F_u^{ab}$ over $F$ has the property that $\text{Hom}(H, \mathbb G_a)$ is a $C$ vector space of dimension the cardinality of $\mathcal A$.
\end{proof}

Lemma \ref{characterbasis} along with \cite[Prop. 2.8, p.86]{ml} imply that the group $U\Pi$, which we know to be free prounipotent, is free prounipotent on a set of cardinality that of $\mathcal A$. We conclude by recording these results:

\begin{theorem}\label{generators} Let $F_u$ be the compositum of all the unipotent Picard--Vessiot extensions of $F$. Then the differential Galois group $U\Pi(F)$ of $F_u$ over $F$ is free prounipotent on a set of cardinality equal to the $C$ vector space dimension of $F/D(F)$.
\end{theorem}

\begin{proof} The only thing remaining to be observed is the cardinality assertion. We have recalled above that for a free prounipotent group $U$ on a set $I$ the cardinality of $I$ is the $C$ vector space  dimension of $\text{Hom}(U, \mathbb G_a)$ 
 \cite[Lemma 2.3, p. 84]{ml}. For $U=U\Pi$,  by Lemma \ref{characterbasis} $\text{Hom}(U\Pi, \mathbb G_a)$ has $C$ dimension  the cardinality of $\mathcal A$.
 \end{proof}
 
If $\text{dim}_C(F/D(F))$ is infinite, then $U\Pi$ is free prounipotent on infinitely many generators. It follows that any unipotent group $U$ is a homomorphic image of $U\Pi$. If $K  \leq U\Pi$ is the kernel of the surjection $U\Pi \to U$ then $E=F_u^K$ is a Picard--Vessiot extension of $F$ with $G(E/F) \cong H$. So we conclude the following about the unipotent Inverse Problem for $F$:

\begin{corollary} \label{unipotentinverse} If $F/D(F)$ is an infinite dimensional $C$ vector space, then every unipotent algebraic group over $C$ occurs as a differential Galois group over $F$.
\end{corollary}

\section{Projective and Free Prounipotent Groups}  \label{S:projfree}

We have used throughout the result \cite[Prop. 2.8, p. 86]{ml} that projective prounipotent groups are free. \emph{A fortiori}, a prounipotent group which is projective as a proalgebraic group is free. And it is elementary to see that a free prounipotent group is projective. We record these observations:

\begin{proposition} \label{P:projectivefree}
Let $U$ be a prounipotent group. Then the following are equivalent:
\begin{enumerate}
\item $U$ is a free prounipotent group\\
\item $U$ is projective as a prounipotent group\\
\item $U$ is projective as a proalgebraic group.
\end{enumerate}
\end{proposition}

\begin{proof} 
Let $U=U(I)$ be a free prounipotent group on the set $I$. In \cite{ml}, prounipotent groups which are projective as prounipotent groups are said to have the \emph{lifting property}. By \cite[Theorem 2.4, p.84]{ml}, to verify the lifting property it suffices to show that for any surjection $\alpha: A \to B$ of unipotent groups with kernel $\mathbb G_a$ and any morphism $f:U(I) \to B$ there is a morphism $g:U(I) \to A$ with $f=\alpha \circ g$. By \cite[Proposition 2.2, p. 83]{ml}, the set $I_0 =\{i \in I \vert f(i) \neq 1\}$ is finite. For each $i \in I_0$, choose $x_i \in A$ such that $\alpha(x_i) =f(i)$. By \cite[Proposition 2.2, p. 83]{ml} again, there is a homomorphism $g:U(I) \to A$ such that $g(i)=x_i$ for $i \in I_0$ and $g(i)=1$ for $i \notin I_0$. Since $\alpha(g(i))=f(i)$ for $i \in I$, again by \cite[Proposition 2.2, p. 83]{ml} $f=\alpha \circ g$. Thus $U(I)$ is projective as a prounipotent group.

Suppose $U$ is a projective prounipotent group.
To show that $U$ is projective as a proalgebraic group, by \cite[Proposition 4, p. 30]{blmm} we need to show that if $\alpha: A \to B$ is a surjection of algebraic groups and $f:U \to B$ is a morphism then there is a morphism $\phi: U \to A$ with $f=\alpha \circ \phi$. Since $U$ is prounipotent, $f(U)$ is a unipotent subgroup of $B$. Since $\alpha$ is surjective, its restriction to the unipotent radical 
$R_u(\alpha^{-1}(f(U))$ is surjective to $f(U)$. Since $U$ is assumed projective in the category of prounipotent groups, there is $\phi_0: U \to R_u(\alpha^{-1}(f(U)$ such that $f=\alpha \circ \phi_0i$. Then $\phi_0$ composed with the inclusion of 
$R_u(\alpha^{-1}(f(U)$ into $A$ is the desired $\phi$. Thus $U$ is projective as a proalgebraic group.

Suppose $U$ is projective as a proalgebraic group. Then it is \emph{a fortiori} projective as a prounipotent group, which means it has the lifting property of \cite{ml}, and hence, as noted above,  by \cite[Proposition 2.8, p.86]{ml} is free prounipotent.
\end{proof}

In \cite[Theorem 2.4, p. 84]{ml}, it is shown that the property of a prounipotent group being projective with respect to all short exact sequences of unipotent groups is equivalent to the property of being projective with respect to those sequences which have $\mathbb G_a$ kernel. As a corollary of the proof of Theorem \ref{T:UPFprojective} we can slightly strengthen that result.

\begin{corollary}\label{C:freetest} A prounipotent group $U$ is free if and only if for every non-split surjective homomorphism $\alpha: A \to B$ of unipotent groups with kernel $K$ isomorphic to $\mathbb G_a$   and for every  surjective homomorphism $f: U \to B$ of prounipotent  groups there is a  surjective homomorphism $\phi: U \to A$ of prounipotent groups such that $f=\alpha \circ \phi$.
\end{corollary} 

\begin{proof} 
\cite[Theorem 2.4, p. 84]{ml} shows that $U$ is projective in the category of prounipotent groups (and hence free) provided there exists a $\phi$ for all $\alpha$ and $f$ as in the corollary without the restrictions that $\alpha$ be non-split and $f$ be surjective. Apply the argument in the proof of Theorem \ref{T:UPFprojective} with $U$ replacing $U\Pi(F)$ to reduce to the cases where $\alpha$ is non-split and $f$ is surjective.
\end{proof}

This paper has been published in \emph{Proceedings of the American Mathematical Society}

https://doi.org/10.1090/proc/15740

\end{document}